\documentclass[10pt]{amsart}
\usepackage{amssymb, amscd, amsmath, amsthm, epsf, epsfig, latexsym}
\def\zz{{\bf Z}}

\def\qq{{\bf Q}}

\def\rr{{\bf R}}

\def\co{\colon\thinspace}
\def\calc{\mathcal{C}}

\def\cs{\mathop{\#}}

\newcommand{\fig}[2] { \includegraphics[scale=#1]{#2} }

\newtheorem{theorem}{Theorem}

\newtheorem{corollary}[theorem]{Corollary}
\newtheorem{prop}[theorem]{Proposition}

\theoremstyle{theorem}
\newtheorem{definition}[theorem]{Definition}

\theoremstyle{definition}

\theoremstyle{definition}

\numberwithin{equation}{section}

\begin{document}
\title{The stable 4--genus of knots }
\author{Charles Livingston}
\thanks{This work was supported by a grant  from the NSF  }
\thanks{\today}
\address{Department of Mathematics, Indiana University, Bloomington, IN 47405}
\email{livingst@indiana.edu}


\maketitle

\begin{abstract} We define the stable 4--genus of a knot $K \subset S^3$, $g_{st}(K)$, to be the limiting value of $g_4(nK)/n$, where $g_4$ denotes the 4--genus and $n$ goes to infinity.  This induces a seminorm on the rationalized knot concordance group, $\calc_\qq = \calc \otimes \qq$.  Basic properties of $g_{st}$ are developed, as are examples focused on understanding the unit ball for $g_{st}$ on specified subspaces of $\calc_\qq$.  Subspaces spanned by torus knots are used to illustrate the distinction between the smooth and topological categories.  A final example is given in which Casson-Gordon invariants are used to demonstrate that $g_{st}(K)$ can be a noninteger.\end{abstract}
\section{Summary.}

In order to better understand the  smooth 4--genus of  knots $K \subset S^3$, denoted $g_4(K)$, we introduce and study here the {\it stable 4--genus}, $$\displaystyle{g_{st}(K) = \lim_{n\to \infty} g_4(nK)/n}.$$  
As will be seen in Section~\ref{algprelims}, the existence of the limit and its basic properties follow from the  subadditivity of $g_4$ as a  function on the classical knot concordance group $\calc$; that is,     $g_4(K \cs J) \le g_4(K) + g_4(J)$ for all $K$ and $J$.  

Neither classical knot invariants nor the invariants that arise from   Heegaard-Floer theory~\cite{os} or   Khovanov homology~\cite{ra} can be used to demonstrate that $g_{st}(K) \notin \zz$ for some $K$.   One result of this paper is the construction of a knots $K$ for which $g_{st}(K)$ is close to $\frac{1}{2}$.   Perhaps of greater interest  is the exploration of the new perspective on the 4--genus and knot concordance offered from the stable viewpoint.  In particular, a number of interesting and challenging new questions arise naturally.  For example, we note that finding a knot $K$ with $0 < g_{st}(K) <\frac{1}{2}$ is closely related to the existence of torsion in $\calc$ of order greater than 2.  We will also consider the distinction between the smooth and topological categories from the perspective of the stable genus. 

\vskip.1in
\noindent{\it Acknowledgements}  Thanks are  due to Pat Gilmer for conversations related to his results on 4--genus, which play a key role in Section~\ref{seconehalf}.   Thanks are also due to  Ian Agol  and  Danny Calegari for discussing with me the analogy between the stable genus and the stable commutator length, described in Section~\ref{secquestions}.   

\section{Algebraic preliminaries.}\label{algprelims}

   The existence of the limiting value and its basic properties are summarized in the following general theorem.

\begin{theorem}\label{limthm} Let $\nu \co G \to \rr_{\ge 0}$ be a subadditive function on an abelian group $G$.  Then:

 \begin{enumerate}
 \item The limit $\nu_{st}(g) = \lim_{n \to \infty} \nu(ng)/n$ exists for all $g \in G$.
 
\item  The function $\nu_{st}\co G \to \rr_{\ge 0}$ is  subadditive and multiplicative: $\nu_{st}(ng) =  n \nu_{st}(g)$ for $n \in \zz_{\ge 0}$.  If $\nu(g) = \nu(-g)$ for all $g$, then $\nu_{st}(g) = \nu_{st}(-g)$ for all $g$.

\item There is a factorization of $\nu_{st}$  through $G_\qq = G \otimes \qq$.  That is, there is a multiplicative, subadditive function $\overline{\nu}_{st} \co G_\qq \to \rr_{\ge 0}$ such that $\nu_{st} = \overline{\nu}_{st} \circ i$ where $i\co G \to G_{\qq}$ is the map $g \to g \otimes 1$.

 \end{enumerate}

\end{theorem}

\begin{proof} The proof of (1) is a standard elementary exercise using the consequence of subadditivity, $\nu(ng) \le n\nu(g)$ for all $g$.  In the  appendix to this paper we summarize a proof.  The rest of the theorem follows easily.  
\end{proof}

A {\it seminorm} on a vector space is a nonnegative multiplicative and subadditive function.  Thus, $\overline{{\nu}}_{st}$ is a seminorm on $G_\qq$.\vskip.1in

\noindent{\bf Notation}  We will usually drop the overbar notation; that is, we will denote both the functions $\nu_{st}$ on $G$ and $\overline{\nu}_{st}$ on $G_\qq$ by $\nu_{st}$ and be clear as to what domain we are using.
\vskip.1in

In our applications we will want to bound $g_{st}$ using homomorphisms on the  concordance group, in particular signatures, the Ozsv\'ath-Szab\'o invariant $\tau$ and the Khovanov-Rasmussen invariant  $s$.  The needed algebraic observation is the following, the proof of which the reader can readily provide.

\begin{theorem}\label{bdthm} If $\sigma \co G \to  \rr$ is a homomorphism and $\nu(g) \ge |\sigma(g)|$ for all $g \in G$, then:

\begin{enumerate}
\item $|\sigma| \co G \to \rr_{\ge 0}$ is subadditive.

\item The stable function ${| \sigma |}_{st} $ satisfies ${| \sigma |}_{st}=  | \sigma |$ and is a seminorm on $G_\qq$.

\item ${\nu}_{st}(x) \ge |\sigma(x)|$ for all $x \in G_{\qq}$.

\end{enumerate}

\end{theorem}

 A seminorm can be completely understood via its unit ball.
 
 \begin{definition} If ${\nu}$ is a seminorm on a vector space $V$, then $B_{{\nu}} = \{x \in V\ |\  {\nu}(x) \le 1\}$.
 \end{definition}

\begin{theorem} Let $\nu$ be a subadditive nonnegative function on an abelian group $G$ and let $\sigma$ be a real-valued homomorphism on $G$.  
\begin{enumerate}
\item  $B_{{\nu}_{st}}$ and   $B_{{|\sigma|}}$ are convex subsets of $G_\qq$.

\item If $\nu(g) \ge |\sigma(g)|$ for all $g \in G$, then  $B_{{\nu}_{st}} \subset B_{{|\sigma|}}$

\end{enumerate}

\end{theorem}


\section{Elementary examples}
We begin exploring the stable genus by computing its value for a few simple examples.  

\subsection{$g_{st}(4_1) = 0$}  The first example of a nonslice knot is the figure eight knot, $4_1$, as originally proved by Fox and Milnor~\cite{fm}.  Since $4_1$ is amphicheiral, $2(4_1)$ is slice, meaning that $g_4(2(4_1)) = 0$.  It follows immediately that  in taking limits, $g_{st}(4_1) = 0$.

\subsection{$g_{st}(3_1) = 1$}  The first knot of infinite order in $\calc$ is the trefoil, $3_1$, as originally proved by Murasugi~\cite{mu}.  Let $\sigma(K)$ denote   of the classical signature of $K$: the signature of $V + V^{\text{\sc T}} $ where $V$ is a Seifert matrix for $K$ and $V^{\text{\sc T}}$ its transpose.  Then we have the Murasugi bound, $g_4(K) \ge \frac{1}{2}| \sigma(K)|$.  Hence   Theorem~\ref{bdthm} applies to show that $g_{st}(3_1) \ge \frac{1}{2} | \sigma(3_1) |= 1$.  On the other hand, $g_4(3_1) = 1$, so $g_{st}(3_1) \le 1$.

\subsection{$g_{st}(3T_{2,7} - 2T_{2,11}) = 2$}  As a final example that illustrates a simple application of Tristram-Levine signatures~\cite{ le2, tr}, we consider the knot $3T_{2,7} - 2T_{2,11}$, where $T_{p,q}$ denotes the $(p,q)$--torus knot.  We will now apply Theorem~\ref{bdthm} to  $\sigma_t$ for appropriate $t$, where  $\sigma_t$ is  the Tristram-Levine signature~\cite{le2, tr}, defined by:
$$\sigma_t(K) = \text{signature}((1 - e^{2 i\pi t})V  +(1 - e^{-2i\pi t})V^{\text{\sc T}}).$$ (Formally, to achieve a concordance invariant  one forms the two-sided limit $\sigma'_t(K) = \lim_{\epsilon \to 0} \frac{1}{2} ( \sigma_{t - \epsilon}(K) + \sigma_{t + \epsilon}(K))$: then $\sigma'$ is a homomorphism on the concordance group for any specific value of $t$.) 
 For the knot $3T_{2,7} - 2T_{2,11}$ this signature function is graphed in Figure~\ref{figsiggraph}.   Since  the function is symmetric about $\frac{1}{2}$, we have graphed the portion of the function on the interval  $[0, \frac{1}{2}]$.

\begin{figure}[h]
\centerline{  \fig{.55}{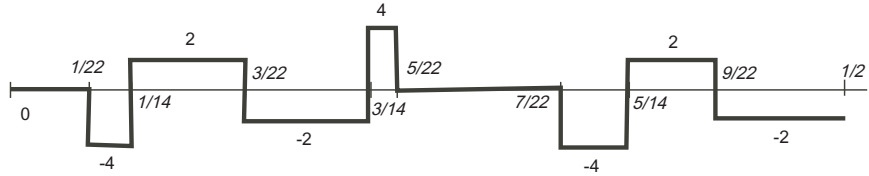}  }
\vskip.1in
\caption{Signature function for $3T_{2,7} - 2T_{2,11}$.} \label{figsiggraph}
\end{figure}

 If we let $x$ be any number between $\frac{3}{14}$ and $\frac{5}{22}$, then   the Tristram-Levine bound $g_{4}(K) \ge  \frac{1}{2}  |\sigma_x(K)|$ implies $g_{st}(K) \ge  \frac{1}{2}  |\sigma_x(K)|$. Thus we have that $g_{st}(3T_{2,7} - 2T_{2,11}) \ge 2$. On the other hand, the reader should have no trouble finding four band moves in the schematic diagram of $3T_{2,7} - 2T_{2,11}$ (Figure~\ref{diagramt23}) that converts it into the torus knot $T_{2,1}$ which is the unknot and in particular bounds a disk.  The corresponding surface in the 4--ball constructed by performing these band moves and capping off with the disk is of genus 2.  Thus $g_{st}(3T_{2,7} - 2T_{2,11}) \le 2$.

  \begin{figure}[h]
   \fig{.65}{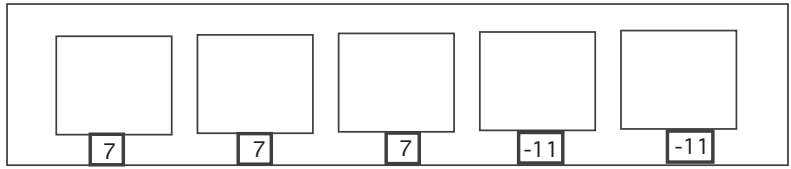} 
   \caption{Schematic diagram for $3T_{2,7} - 2T_{2,11}$}\label{diagramt23}
   \end{figure}

\section{Families of knots:  $xT_{2,7} + yT_{2,11}$.}

A nice illustrative example is given by restricting to the subspace $S$ of $\calc_\qq$ spanned by the torus knots $   T_{2,7}$ and $  T_{2,11}$.    We want to understand the unit ball of $g_{st}$ on $S$ in terms of the unit ball associated to the function Max$_{0\le t \le 1} \{\sigma_t\}$; for any particular example  it is more straightforward to directly analyze the signature function.  In the present case, the signature functions  for $T_{2,7}$ and $T_{2,11}$ are zero near $t = 0$ and increase by two at each of the jumps at the points $\{1/14, 3/14, 5/14\}$ and $\{1/22, 3/22, 5/22, 7/22, 9/22\}$, respectively.  For the readers convenience, we order the union of these two sets:  $$\{1/22, 1/14, 3/22, 3/14, 5/22, 7/22, 5/14, 9/22\}.$$  Evaluating the signature functions $\sigma_t(xT_{2,7} + yT_{2,11})$ at values between each of  these numbers and for some $t$ close to $\frac{1}{2}$  yields the following set of inequalities:

\begin{eqnarray*}
g_{st}(xT_{2,7} + yT_{2,11}) & \ge & |y|\\
g_{st}(xT_{2,7} + yT_{2,11})& \ge & | x + y|\\
g_{st}(xT_{2,7} + yT_{2,11}) & \ge & |x +2y|\\
g_{st}(xT_{2,7} + yT_{2,11})& \ge &  |2x + 2y|\\
g_{st}(xT_{2,7} + yT_{2,11}) & \ge &  |2x + 3y|\\
g_{st}(xT_{2,7} + yT_{2,11}) & \ge &  |2x + 4y|\\
g_{st}(xT_{2,7} + yT_{2,11})& \ge &  |3x + 4y|\\
g_{st}(xT_{2,7} + yT_{2,11})& \ge &  |3x + 5y|.\\
\end{eqnarray*}
 Based on these, we find that the unit ball $B_{{g}_{st}}$ restricted to the span of $T_{2,7}$ and $T_{2,11}$ is contained in the set illustrated in Figure~\ref{ballgraph}.

\begin{figure}[h]
\centerline{  \fig{.45}{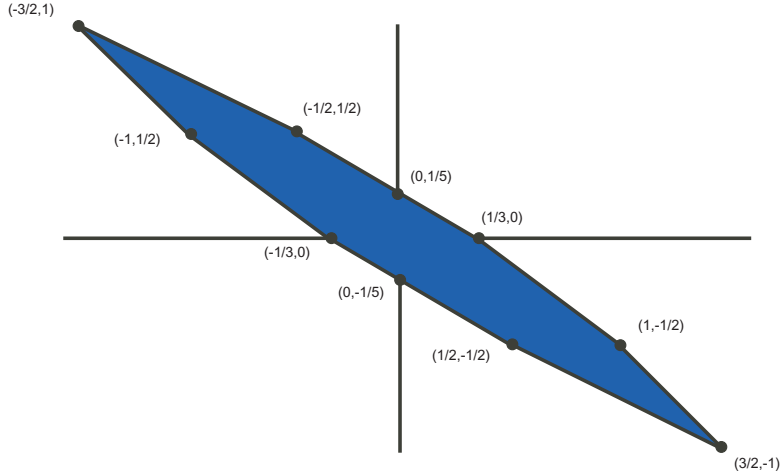}  }
\vskip.1in
\caption{The unit ball for $g_{st}$ in the span of $T_{2,7}$ and $T_{2,11}$.} \label{ballgraph}
\end{figure}

By convexity, to show that this set is actually the unit ball for $g_{st}$, we  need to check only the vertices.  For instance, we want to see that $g_{st}(\frac{3}{2}T_{2,7} - T_{2,11}) = 1$.  That is, we need to show $g_{st}(3T_{2,7} - 2T_{2,11}) = 2$.  That calculation was done in the previous section.  The other vertices are handled similarly.  (That is, one shows that  $g_4(T_{2,7}) = 3$, $g_4(T_{2,7} - T_{2,11}) = 2$, $g_4(3T_{2,7} - 2T_{2,11}) = 2$ and $g_4(2T_{2,7} - T_{2,11}) = 2$.  The point $(0, \frac{1}{5})$ is not a vertex so need not be considered.)

\vskip.1in
\noindent{\bf Note.}  Rick Litherland~\cite{lith} has   proved that for any pair of two stranded torus knots, the 4--genus of a linear combination $xT_{2,k} + yT_{2,j}$ is determined by its signature function.


\section{A smooth versus topological comparison: $xT_{3,7} + yT_{2,5}$.}
We now summarize a more complicated example of the computation of the $g_{st}$ unit ball on the 2--dimensional subspace spanned by $T_{3,7}$ and $T_{2,5}$.  The added complexity occurs because the signature function of $T_{3,7}$ does not determine its smooth 4--genus; this signature function has positive jumps at $1/21, 2/21,  4/21, 5/21$ and $ 8/21$ but a negative jump at  $   10/21$.  Thus its maximum value is 5, and its value at $t=\frac{1}{2}$ is 4.  On the other hand, the Ozsv\'ath-Szab\'o $\tau$ or Khovanov-Rasmussen invariant $s$ both take value 6, and thus determine the smooth 4--genus of $T_{3,7}$ to be 6.  (See~\cite{os, ra} for details.)

Considering only the signature function, we can show that the  unit  $g_{st}$ ball is contained within  the entire shaded region.  Using either $\tau$ or $s$ places additional bounds which eliminate the two thin darker triangles.  The innermost parallelogram represents points that we know are in the unit ball.

   \begin{figure}[h]
   \fig{.4}{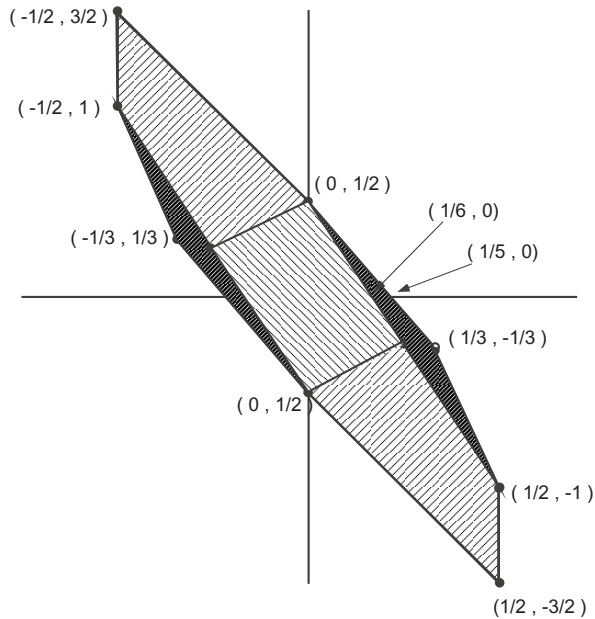} 
   \caption{Bounds on the topological and smooth unit ball for $g_{st}$ on the span $\left< T_{3,7} + yT_{2,5} \right>$.}\label{smooth-top}
   \end{figure} 

\vskip.1in

\noindent{\bf Note.}  Recent work has slightly enlarged the region which we know lies in the unit $g_{st}$ ball for the span of these two knots, but most of the region remains unknown.  We know of no knots in this span for which the topological and smooth 4--genus differ.


\section{A 4--dimensional example.}

As our final example  related to finding a $g_{st}$ unit ball, we consider the span of the first four  knots that are of infinite order in $\calc$:  $3_1$, $5_1$, $5_2$ and $6_2$.  If we identify the span of these with $\qq^4$ via the coordinates $x_1(3_1) + x_2(5_1) + x_3(5_2) +x_4(6_2)$, then the unit ball determined by the maximum of the signature function turns out to be a polyhedron formed as the convex hull of 24 points that come in antipodal pairs.  We list one from each pair:

\begin{enumerate}

\item $   (2, -1,0,0),   (0,1,-2,0),  (0,1,0,-1),  (2,-1,0,-1), (0,0,1,0),    (2,0,-1,0) $ \vskip.05in 

\item $ (0,1,0,-2)$ \vskip.05in

\item $ (2,1,-2,-2), (2,1,-2,-1), (0,1,-2,1) , (0,0,1,-2), (2,0,-1,-2)   $.  \vskip.05in
 
\end{enumerate}

Those in the first set of five have all been shown to have $g_4 = 1$.  For those in the last set we have been unable to compute the genus or stable genus.  For the second set, $(0,1,0,-2) $, we have been unable to compute the 4--genus, but we know that twice this knot has 4--genus 2, and hence its stable 4--genus is 1.  


\section{A knot with $g_{st}(K)$ near $ \frac{1}{2}$. Gilmer, Casson-Gordon bounds.}\label{seconehalf}

We begin by presenting  Gilmer's result~\cite{gi} bounding the 4--genus of a knot $K$ in terms of Casson-Gordon signature invariants~\cite{cg1}.  

Let $K$  be a  knot and let $M_d(K)$ denote its $d$--fold branched cover, with $d$ a prime power.   To each prime $p$ and character $\chi \co H_1(M_d(K), \zz) \to \zz_p$, there is the   Casson-Gordon invariant $\sigma(K,\chi) \in \qq$.    By~\cite{gi}, this invariant is additive under connected sum of knots and direct sums of characters.   A  special case of the main  theorem of~\cite{gi3} states the following:

\begin{theorem} If $K$ is an algebraically slice knot for which   $H_1(M_d(K), \zz)  \cong \zz_p^{2n}$ and $g_4(K) = g$, then  there is a subspace $H \subset \text{Hom}(H_1(M_d(K), \zz), \zz_p) \cong H^1(M_d(K), \zz_p)$ of dimension $\frac{1}{2} (2n-2(d-1)g) $ such that for all $\chi \in H$, $|\sigma(K, \chi) |\le 2dg$.
\end{theorem}

It was observed in~\cite{gl} that $H$ can be assumed to be invariant under the deck transformation.  Applying this and specializing to the case of $d = 3$, we have:

\begin{corollary}\label{corgilmer} If $K$ is an algebraically slice knot for which  $H_1(M_3(K), \zz)  \cong \zz_p^{2n}$ and $g_4(K) = g$, then  there is a $\zz_3$--invariant subspace $H \subset \text{Hom}(H_1(M_3(K), \zz), \zz_p) \cong H^1(M_3(K), \zz_p)$ of dimension $  n-2g $ such that for all $\chi \in H$, $|\sigma(K, \chi) |\le 6g$.

\end{corollary}

\noindent{\bf Example} Consider the knot illustrated in Figure~\ref{genknot}, which we denote $K(J_1, J_2)$.  This family of knots has been used throughout the study of knot concordance; a detailed description can be found, for instance, in~\cite{gl}, in which the details  of the results we now summarize can be found. First, the homology of the 3--fold branched cover is
  the direct sum of cyclic groups of order seven: $H_1(M_3(K(J_1, J_2))) \cong \zz_7 \oplus \zz_7$.  Furthermore, 
   the homology splits as the direct sum of $E_2 \cong \zz_7$
    and $E_4 \cong \zz_7$, the $2$--eigenspace  
   and $4$--eigenspace of the deck transformation.  (Note that $2^3 = 4^3 = 1 \mod 7$.) 
   
    Similarly, $H^*_1(M_3(K(J_1,J_2))) = \text{Hom}( H_1(M_3(K(J_1,J_2))) , \zz_p)$ splits as a direct sum of eigenspaces, which we denote $E^*_2$ and $E^*_4$.
  Using two  eigenvectors as a basis for $H^*_1(M_3(K(J_1,J_2)))$ and letting $\chi_{a,b}$ be the character corresponding to $(a,b)$ via this identification, as proved in~\cite{gl} we have:

\begin{theorem} \label{thmcg} $\sigma(K(J_1, J_2), \chi_{a,0}) =   \sigma_{a/7}(J_1) +   \sigma_{2a/7}(J_1)+ \sigma_{4a/7}(J_1)$; similarly, $\sigma(K(J_1, J_2), \chi_{0,b}) =   \sigma_{b/7}(J_2) +   \sigma_{2b/7}(J_2)+ \sigma_{4b/7}(J_2)$.  In particular, it follows that  $\sigma(K(J_1, J_2), \chi_{0,0}) = 0$.
 \end{theorem}

\begin{figure}[h]
\centerline{  \fig{.6}{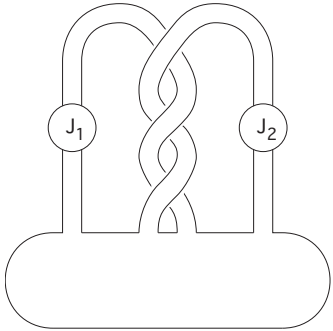}  }
\vskip.1in
\caption{The knot $K(J_1,J_2)$.} \label{genknot}
\end{figure}

We can now demonstrate that  particular knots in this family have $g_{st}(K(J_1, J_2))$ near $ \frac{1}{2}$.

\begin{theorem} For any $\epsilon >0$, there is a knot $J$ so that $\frac{1}{2} (1-\epsilon)\le g_{st}(K(J,-J)) \le\frac{1}{2}$.

\end{theorem}

\begin{proof} By the additivity of $3$--genus,  for any knot $J$ we have $g_3(2K(J,-J)) = 2$.  On the evident Seifert surface for  $2K(J,-J) $ there is a curve on the surface with framing 0 representing the knot $J \cs -J$, which is slice.  Thus, the Seifert surface can be surgered in the 4--ball to give a surface of genus one bounded by $2K(J,-J)$.  Therefore, $g_4(2K(J, -J)) \le 1$ and $g_{st}(K(J,-J)) \le \frac{1}{2}$.

We now proceed to show that for each  $\epsilon$ there is some $J$ for which $g_{st}(K(J,-J)) \ge \frac{1}{2}(1 - \epsilon)$.  For a given $J$, if this is inequality is false, then for  some $n >0$, $g_4(nK(J,-J)) < \frac{1}{2}(1-\epsilon)n$.  (Since this holds for some $n$, it holds for all $n$ sufficiently large.)  For this $n$, we have $H_1(M_3(nK(J,-J))) = \zz_7^{2n}$. Applying Corollary~\ref{corgilmer} we find the relevant subgroup $H$ has dimension $\dim (H) >  n - 2(\frac{1}{2}(1-\epsilon)n).$
Simplifying, we have $\dim (H) >  \epsilon n$.

  Since $H_1( M_3(K(J,-J)))$ splits as the direct sum of a $2$--eigenspace and a $4$--eigenspace, the same is true for $ H_1(M_3(n K(J,-J)))$.  Thus, we also have an eigenspace splitting of $H^*_1(M_3(n K(J,-J)))$. The subspace $H$ given by Corollary~\ref{corgilmer} is invariant under the deck transformation, so it too must split as the sum of eigenspaces,  $ H = H_2 \oplus H_4$.  Given that  $\dim (H) >  \epsilon n$, one of these must have dimension at least $\frac{1}{2} \epsilon n$.  We will assume   $\dim (H_2) > \frac{1}{2}  \epsilon n$; the case  $\dim (H_4) > \frac{1}{2}  \epsilon n$ is similar.

We next use the fact, easily established using the Gauss-Jordan algorithm, that a subspace of dimension $a$ in $\zz_p^b$ contains some vector with at least $a$ nonzero coordinates.  Thus, $H_2$ contains a vector $h$ with at least  $\frac{1}{2}  \epsilon n$ nonzero coordinates.

For the character $\chi_h$ given by $h$, by the additivity of Casson-Gordon invariants and Theorem~\ref{thmcg}, $$\sigma(K, \chi_h) =\sum \sigma(K(J, -J), \chi_{a_i,0}) = \sum \sigma_{\frac{a_i}{7}}(J)$$ where the sum has at least $ \frac{1}{2}  \epsilon n$ elements and each $a_i = 1, 2, $ or $3$.  
Now, letting $M> 0 $ be a fixed constant assume that $\sigma_{\frac{a_i}{7}}(J) > M$ for $a_i = 1, 2, 3$.  Such a $J$  is easily constructed  using the connected sum of $(2,k)$-torus knots.  Then $\sigma(nK(J,-J), \chi_h)  > \frac{1}{2} \epsilon n M$.  Thus, we will have a contradiction to Corollary~\ref{corgilmer}  if $\frac{1}{2} \epsilon n M \ge 6( \frac{1}{2}(1-\epsilon)n) $.  Simplifying, we find that there is a contradiction if $M \ge 6   (\frac{1-\epsilon}{\epsilon}) $.  In conclusion, if $\sigma_\frac {a}{7} \ge    6   (\frac{1-\epsilon}{\epsilon}) $ then $g_{st}(K) \ge (1-\epsilon)\frac{1}{2}$.  

In the case that we are working with the $4$ --eigenspace instead of the $2$--eigenspace, the same condition appears, since Corollary~\ref{corgilmer} concerns the absolute value of the Casson-Gordon invariant, and switching eigenspaces simply interchanges $\sigma_\frac{a}{7}(J)$ with  $\sigma_\frac{a}{7}(-J)$.

\end{proof}

\subsection{Other non-integer examples.} The knot $K(J,-nJ)$ illustrated in Figure~\ref{genknot2} can be shown to satisfy $g_4((n+1) K(J,-nJ)) \le \frac{n}{n+1}$, in much the same way as the special case of $n=1$, $K(J,-J)$.  Thus, $g_{st}(K(J,-nJ) \le    \frac{n}{n+1}$.  The argument used above, based on the 3--fold cover, cannot be successfully applied to find a lower bound.  However, using the 2--fold cover we have been able to prove a weaker result.  Given $n$, there is a $J$ so that $\frac{n-1}{n} \le g_{st}(K(J,-nJ)) \le  \frac{n}{n+1}$.

\begin{figure}[h]
\centerline{  \fig{.6}{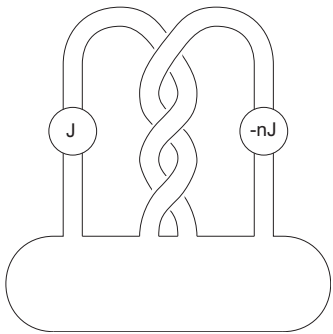}  }
\vskip.1in
\caption{The knot $K(J_1,J_2)$.} \label{genknot2}
\end{figure}

\section{Question.}\label{secquestions}

\begin{enumerate}

\item  Is $g_{st}$ a norm on $\calc_\qq$?  That is, if $g_{st}(K) = 0$, does $K$ represent torsion in $\calc$?\vskip.1in

\item  Is there a knot $K$ such that $0 < g_{st}(K) < \frac{1}{2}$?  This question relates to that of finding torsion of order greater than 2 in $\calc$.  For instance, if there is a knot $K$ of order three, then $g_4(3K) =0$.    A simpler question than that of finding such a knot is to find a knot   satisfying $g_4(3K) = 1$ but $g_4(2K) \ge 2$.
\vskip.1in

\item Is $g_{st}(K) \in \qq$ for all $K$?  Presumably the examples constructed in the previous section satisfy $g_{st} = \frac{n}{n+1}$   for some $n$, though this seems difficult to prove.\vskip.1in

\item Related to this previous question, is there a knot for which $g_{st}(K) \ne g_4(nK)/n$ for any $n$?  \vskip.1in

\item  Let $\{K_i\}$ be finite set of knots and let $S$ be the span of these knots in $\calc_\qq$.  Is the $g_{st}$ ball in $S$ a finite sided polyhedron? \vskip.1in

\item For some pair of distinct nontrivial positive torus knots, $T_{p,q}$ and $T_{p', q'}$, with $p,  q, p', q' >2$, determine the unit $g_{st}$ ball on their span in $\calc_\qq$, in either the smooth or topological category.

\end{enumerate}

\subsection{Stable commutator length}  If $g \in [G,G]$ is an element in the commutator subgroup of a group $G$, it can be expressed as a product of commutators.  The shortest such expression for $g$ is called the commutator length, $cl(g)$.    The limit $\lim_{n \to \infty}  {cl(g^n)}/{n}$ is called the stable commutator length.  The notion was first studied in~\cite{ba}.  Although no formal connections between this and the stable 4--genus are known at this time, the possibility of such connections is provocative.  We note that Calegari's work~\cite{cal} has revealed much of the behavior of the stable commutator length for  free groups. In particular, the stable commutator length is always rational for free groups, though this is not true for all groups~\cite{zhu}.  Further details can be found in~\cite{cal2}


\appendix

\section{Limits}\label{appendix}  We sketch the proof of Theorem~\ref{limthm}, restated as follows.

\begin{prop} Let $f \co \zz_+ \to R_{\ge 0}$ satisfy $f(nm) \le nf(m)$ for all $n$ and $m$.  Then $\lim_{n \to \infty} f(n)/n$ exists.

\end{prop}

\begin{proof} Let $L$ be the greatest lower bound of $\{f(n)/n\}_{n \in \zz_+}$. For any $\epsilon$ there is an $N$ such that $f(N)/N \le L + \frac{\epsilon}{2}$.   Any $n$ can be written as $n = aN + b$ where $0\le b <N$.  Also, $f(b) \le B = \max\{f(b)\}_{0 \le b < N}$.  By subadditivity we have  $f(n) \le af(N) + f(b)$.  Dividing by $n$ we have
$$ \frac{f(n)}{n} \le \frac{af(N)}{aN + b} + \frac{f(b)}{aN + b}  \le \frac{ f(N)}{ N  } + \frac{B} { aN }.$$ Thus, if $n$ is chosen large enough that $\frac{B}{aN} \le \frac{\epsilon}{2}$, (for instance, choose  $n \ge \frac{2B}{\epsilon} + N$) we have $f(n)/n \le L + \epsilon$.

\end{proof}


 \end{document}